\PassOptionsToPackage{dvipsnames}{xcolor}
\documentclass[onefignum,onetabnum]{siamart220329}

\ifpdf
\hypersetup{
  pdftitle={An Adaptive Covariance Parameterization Technique for the Ensemble Gaussian Mixture Filter},
  pdfauthor={Andrey A Popov, and Renato Zanetti}
}
\fi

\usepackage{hyperref}

\usepackage{mathtools}
\usepackage{amsmath}
\usepackage{amssymb}

\usepackage{algorithm}
\usepackage{algpseudocode}

\usepackage{cleveref}

\usepackage{tikz}
\usetikzlibrary{calc}
\usetikzlibrary{math}

\newcommand\sqdiag{\tikz{\draw (0,0) rectangle (0.7em, 0.7em); \draw (0, 0.7em) -- (0.7em, 0);}} 
\DeclareMathOperator{\bdiag}{\sqdiag}

\usepackage{pgfplots}
\pgfplotsset{compat=1.17} 
\usepackage{pgfplotstable}

\usepackage{cbcolors}

\DeclareMathOperator{\tr}{tr}

\newsiamremark{remark}{Remark}
\newsiamremark{procedure}{Procedure}
\crefname{procedure}{Procedure}{procedures}
\newsiamthm{claim}{Claim}

\title{An Adaptive Covariance Parameterization Technique for the Ensemble Gaussian Mixture Filter\thanks{Submitted to the arXiv \today.\funding{This work was sponsored in part by DARPA (Defense Advanced Research Projects Agency) under STTR contract number W31P4Q-21-C-0032.}}}

\author{Andrey A Popov\thanks{Oden Institute for Computational Engineering and Sciences, University of Texas at Austin
  (\email{andrey.a.popov@utexas.edu}).}
\and Renato Zanetti\thanks{Department of Aerospace Engineering \& Engineering Mechanics, University of Texas at Austin
  (\email{renato@utexas.edu}).}}

\usepackage[listings,skins,breakable]{tcolorbox}
{%
\end{tcolorbox}\endgroup}

\newcommand\1{1}
\def\!#1{\mathcal{#1}}
\def\*#1{\boldsymbol{\mathbf{#1}}}
\def\|#1{\textnormal{#1}}
\def\##1{\mathfrak{#1}}

\def\norm#1{\left\lVert#1\right\rVert}

\DeclareMathOperator*{\argmax}{arg\,max}

\begin{document}

\maketitle

\begin{abstract}
The ensemble Gaussian mixture filter combines the simplicity and power of Gaussian mixture models with the provable convergence and power of particle filters.
The quality of the ensemble Gaussian mixture filter heavily depends on the choice of covariance matrix in each Gaussian mixture.
This work extends the ensemble Gaussian mixture filter to an adaptive choice of covariance based on the parameterized estimates of the sample covariance matrix.
Through the use of the expectation maximization algorithm, optimal choices of the covariance matrix parameters are computed in an online fashion. 
Numerical experiments on the Lorenz '63 equations show that the proposed methodology converges to classical results known in particle filtering. Further numerical results with more advances choices of covariance parameterization and the medium-size Lorenz '96 equations show that the proposed approach can perform significantly better than the standard EnGMF, and other classical data assimilation algorithms.
\end{abstract}



\section{Introduction}
Sequential data assimilation~\cite{asch2016data,reich2015probabilistic} aims to perform Bayesian inference on the state of some natural process from an inaccurate computational model and sparse and noisy observations.
Traditionally, particle filter methods have been viewed as theoretically nice, but practically useless for inference of high dimensional systems. 
Recent advances in particle filters for high dimensions~\cite{van2019particle} have challenged this view.

The ensemble Gaussian mixture filter (EnGMF)~\cite{anderson1999monte, liu2016efficient, yun2022kernel} and its related cousin the adaptive Gaussian mixture filter (AGMF)~\cite{stordal2011bridging, van2019particle} are sequential data assimilation algorithms that extend the idea of Gaussian mixture model (GMM) kernel density estimation (KDE) to the front of data assimilation. 
The EnGMF is based on the observation that Gaussian mixture models, under linear observation assumptions,  are closed under multiplication~\cite{anderson2012optimal}.
The quality of the inference begotten by the EnGMF is directly related to the accuracy of the GMM assumption about the prior distribution. 
This prior distribution is typically determined in whole by Monte Carlo samples through choices of the means and covariances of the GMM. 
While the choice of means is readily apparent as that of the Monte Carlo samples, the choice of covariance is less straightforward, and is the subject of much of the research surrounding KDE~\cite{silverman2018density, breiman1977variable}.

The key innovations of this work are as follows: 
We first provide theoretical results that show the convergence of the EnGMF for a certain class of probabilistic assumption on the bandwidth parameter in the classical EnGMF algorithm.
we next generalize the EnGMF by introducing the parameterization of statistical covariance matrix estimates from other ensemble-based filters to the EnGMF. 
We finally show how the EnGMF machinery could be used to choice the value of these parameters in an optimal adaptive fashion by utilizing the expectation maximization algorithm.
Thus the sum total of these results is the adaptive Gaussian mixture filter (AEnGMF) which utilizes all this machinery for inference.

This work is organized as follows: we first introduce the data assimilation problem and the EnGMF in~\cref{sec:background}. We next present the adaptive ensemble Gaussian mixture filter in~\cref{sec:AEnGMF} along with the expectation maximization algorithm in~\cref{sec:EM}. Numerical experiments are provided in~\cref{sec:numerical-experiments}, and concluding remarks in~\cref{sec:conclusion}.

\section{Background}
\label{sec:background}

Given a model that evolves a natural process of interest from time index $i-1$ to time index $i$,
\begin{equation}\label{eq:model}
    x^t_i = \!M(x^t_{i-1}) + \xi_i,
\end{equation}
with model error $\xi_i$, the goal is to estimate the true state $x^t$ of said process given some non-linear observation,
\begin{equation}\label{eq:observations}
    y_i = \!H(x^t_i) + \eta_i,
\end{equation}
with observation operator $\!H$ and an additive error term $\eta_i$. Denote with $Y_i$ all the observations up to and including time index $i$,
\begin{equation}
    Y_i = \{y_1,\ y_2,\ ...\ y_i\},
\end{equation}
given a prior at time index $i$, namely $x^b_i = x^t_i | Y_{i-1}$, we aim to perform Bayesian inference on these two sources of information,
\begin{equation}\label{eq:Bayesian-inference}
    p(x^b_i | y_i) \propto p(x^b_i)\ p(y_i | x^b_i)
\end{equation}
resulting in the `analysis', $x^a_i = x^b_i | y_i = x^t_i | Y_i$.

For the remainder of this paper, we assume that the model error $\xi_i$ is always zero, and thus the model~\cref{eq:model} is exact.

We next describe how a solution to~\cref{eq:Bayesian-inference} can be achieved using Monte Carlo sampling and the EnGMF.

\subsection{The Ensemble Gaussian Mixture Filter}
Assume that we have a collection of $N$ particles at time index $i$ that is represented as $\*X^b_i = [x^b_{i,1}, x^b_{i,2}, \dots x^b_{i,N}]$ that are weighted samples from the prior distribution $p(x^b_i)$ with weights $\{u_{i,j}\}_{j=1}^N$. Given non-linear observations of the truth~\cref{eq:observations}, our aim is to find a collection of $N$ particles $\*X^a_i = [x^a_{i,1}, x^a_{i,2}, \dots x^a_{i,N}]$ that are samples from the posterior distribution, such that the posterior is the prior conditioned by the observations,
\begin{equation}
    x^a_i = x^b_i | y_i,
\end{equation}
solving the Bayesian inference problem~\cref{eq:Bayesian-inference}.

As the prior distribution of the particles is unknown, an assumption about this distribution can be made. 
From kernel density estimation theory, the ensemble Gaussian mixture filter (EnGMF) assumes that the distribution of the prior state at time index $i$, $x^b_i$ is given by the Gaussian mixture,
\begin{equation}\label{eq:prior-Gaussian-mixture}
    x^b_i \sim \sum_{j=1}^N u_{i,j}\,\!N\left(\bar{x}^b_{i,j},\, \*B^b_{i,j}\right),
\end{equation}
where each mean exactly corresponds to one of the particles in the ensemble,
\begin{equation}
    \bar{x}^b_{i,j} \coloneqq x^b_{i,j}.
\end{equation}

The observation distribution at time index $i$ is given by the Gaussian mixture,
\begin{equation}\label{eq:observation-Gaussian-mixture}
    y_i|x^t_i \sim \sum_{k=1}^M v_{i,k}\,\!N\left(\bar{y}_{i,k},\, \*R_{i,k}\right),
\end{equation}
which is a generalization of the typical Gaussian assumptions on the observation error made in data assimilation literature.

The posterior distribution at time index $i$ is defined by the Gaussian mixture,
\begin{equation}\label{eq:posterior-Gaussian-mixture}
    x^a_i \sim \sum_{j=1}^N\sum_{k=1}^M w_{i,j,k}\,\!N\left(\bar{x}^a_{i,j,k},\, \*B^a_{i,j,k}\right),
\end{equation}
with the following set of definitions,
\begin{equation}\label{eq:posterior-update}
\begin{aligned}
    \bar{x}^a_{i,j,k} &= \bar{x}^b_{i,j} - \*G_{i,j,k}\left(\!H(\bar{x}^b_{i,j}) - \bar{y}_{i,k}\right),\\
    \*B^a_{i,j,k} &= \left(\*I -  \*G_{i,j,k}\*H_{i,j}^T\right)\*B^b_{i,j},\\
    \*G_{i,j,k} &= \*B^b_{i,j}\*H_{i,j}^T\left(\*H_{i,j}\*B^b_{i,j}\*H_{i,j}^T + \*R_{i,k}\right)^{-1},\\
    w_{i,j,k} &\propto u_{i,j} v_{i,k}\, \!N\left(\bar{y}_{i,k}\,\middle|\, \!H(\bar{x}^b_{i,j}),\, \*H_{i,j}\*B^b_{i,j}\*H_{i,j}^T + \*R_{i,k}\right),\\
    \*H_{i,j} &= \left.\frac{d \!H}{d x}\right\rvert_{x = \bar{x}^b_{i,j}}.
\end{aligned}
\end{equation}
where $\bar{x}^a_{i,j,k}$, are the analysis Gaussian mixture means,  $\*B^a_{i,j,k}$ are the analysis Gaussian mixture covariances, $\*G_{i,j,k}$ is similar to a gain matrix, $w_{i,j,k}$ are the Gaussian mixture weights, and $\*H_{i,j}$ is the linearization of the observation operator around $\bar{x}^b_{i,j}$.
When the observation operator is linear, $\!H(x) = \*H x$, the posterior GMM~\cref{eq:posterior-Gaussian-mixture} is exactly the posterior corresponding to the assumed prior~\cref{eq:prior-Gaussian-mixture} and the observation~\cref{eq:observation-Gaussian-mixture} distributions.

Each Gaussian distribution in~\cref{eq:prior-Gaussian-mixture} has the following probability density function,
\begin{equation}\label{eq:Gaussian}
    \!N(x | \bar{x}^b_{i,j}, \*B^b_{i,j}) = \left\lvert 2 \pi\*B^b_{i,j}\right\rvert^{-\frac{1}{2}} e^{-\frac{1}{2}\norm{x - \bar{x}^b_{i,j} }_{\*B^{b,-1}_{i,j}}},
\end{equation}
with the other Gaussian distributions in \cref{eq:observation-Gaussian-mixture,eq:posterior-Gaussian-mixture,eq:posterior-update} having a similar form.

\begin{remark}[Normalization Factors]\label{rem:normalizing-factor}
Note that when either the observation operator $\!H$ is non-linear, or the kernel covariance matrices $\*B_{i,j}$ are not identical, extra care must be taken when computing the weights $w_{i,j,k}$ in \cref{eq:posterior-update}, as the covariances $\*H_{i,j}\*B^b_{i,j}\*H_{i,j}^T + \*R_{i,k}$ are not necessarily equal. This means that, the normalization factors in each Gaussian term (similar to \cref{eq:Gaussian}), 
\begin{equation}\label{eq:weight-normalization-factor}
    \left\lvert 2 \pi\left( \*H_{i,j}\*B^b_{i,j}\*H_{i,j}^T + \*R_{i,k} \right)\right\rvert^{-\frac{1}{2}},
\end{equation}
are required to be computed. This can be performed in a computationally efficient manner through the use of the Cholesky decomposition and the log-sum-exp trick~\cite{blanchard2021accurately}.
\end{remark}

While the transformation of the distribution in~\cref{eq:posterior-Gaussian-mixture} results in an estimate of the posterior distribution, the means $\bar{x}^a_{i,j}$ of this distribution are not actually samples from this distribution, thus it is not the case that the posterior samples are equivalent to these means,
\begin{equation}
    x^a_{i,j} \not= \bar{x}^a_{i,j}.
\end{equation}
A resampling procedure is therefore required in order to obtain independently and identically distributed (iid) samples from~\cref{eq:posterior-Gaussian-mixture}.
What follows is one such procedure.

\begin{procedure}[EnGMF resampling]\label{pro:resampling-procedure}
Given the final posterior Gaussian mixture distribution in \cref{eq:posterior-Gaussian-mixture}, it is possible to resample $S$ samples from the posterior GMM through the following procedure:
\begin{enumerate}
    \item for $s=1,\dots, S$, $\mathbf{X}^a_{i,s}$, sample the random variable $\ell$ from the discrete distribution defined by the weights $\{w_{i,j,k}\}_{j,k}$,
    \item sample $\mathbf{X}^a_{i,s}$ from the Gaussian $\!N\left(x | \bar{x}^a_{i,\ell}, \*B^a_{i,\ell}\right)$,
\end{enumerate}
enabling samples to be generated from the posterior.
\end{procedure}

\begin{remark}[Arbitrary Sampling of the Posterior]\label{rem:posterior-arbitrary-sampling}
Note, that using~\cref{pro:resampling-procedure} we are able to arbitrarily sample from the posterior distribution. This means that the number of posterior samples $S$ could be significantly larger or significantly smaller than the original number of samples $N$ used to generate said posterior.
\end{remark}

\begin{remark}[Independent and Identically Distributed Samples]
    Note that while we make the convenient assumption that the samples generated by~\cref{pro:resampling-procedure} are iid, this is not actually the case. The parameters of each mode of the GMM are actually functions of the prior samples, and are themselves random variables, and thus introduce a dependence if two samples come from the same mode.
\end{remark}

\begin{remark}[Prior Uniform Weights]
If \cref{pro:resampling-procedure} is utilized to re-sample the particles at every step of the assimilation, then the prior distribution weights in \cref{eq:prior-Gaussian-mixture} are all uniform $u_{i,j} = \frac{1}{N}$ under the assumption of a uniform transition density.
\end{remark}

\begin{remark}[Differences Between the EnGMF and the AGMF]
   Unlike the EnGMF, in the AGMF (see \cite{stordal2011bridging, van2019particle}), instead of resampling like in~\cref{pro:resampling-procedure}, the weights are scaled by a defensive factor towards uniformity,
   \begin{equation}
        w_{i,j,k} = \alpha_i w_{i,j,k} + (1 - \alpha_i)\frac{1}{N},   
   \end{equation}
   with a `defensive factor', $\alpha_i$, which ensures that the weights do not degenerate.
   The new particles are taken to be the means of the candidate posterior distribution~\cref{eq:posterior-Gaussian-mixture}.
   This idea is not explored in this work.
\end{remark}

The posterior GMM~\cref{eq:posterior-Gaussian-mixture} is only a good representation of the exact posterior if the prior GMM assumption~\cref{eq:prior-Gaussian-mixture} is a good approximation of the distribution that originated the particles $\*X_i^b$.
For a finite ensemble $N$ it is possible that the EnGMF analysis is a poor representation of the truth. 
The prior GMM assumption~\cref{eq:prior-Gaussian-mixture}, in the particle limit $N\to\infty$, converges to the exact prior under certain assumption on the covariances, $\*B_{i,j}$, which is a known result from kernel density estimation literature~\cite{silverman2018density}. 
We now show that under some assumptions on the covariances, $\*B_{i,j}$, the posterior GMM~\cref{eq:posterior-Gaussian-mixture} begotten from the EnGMF procedure also converges, in the particle limit, $N\to\infty$, to a distribution obtained from performing exact Bayesian inference.

\begin{theorem}[EnGMF convergence]\label{thm:EnGMF-convergence}
    Assuming the observation distribution is exact~\cref{eq:observation-Gaussian-mixture},
    if the means $\bar{x}^b_{i,j}$ in the estimated prior distribution GMM~\cref{eq:prior-Gaussian-mixture} are samples from the underlying exact distribution with weights $u_{i,j}$, and the prior Kernel covariance matrices tend to zero in the limit of ensemble size, $\lim_{N\to\infty}\*B^b_{i,j} = 0$, then the EnGMF with the resampling procedure~\cref{pro:resampling-procedure} converges to a filter in the class of sequential importance resampling (SIR) filters.
\end{theorem}
\begin{proof}
Given the assumptions above, in the limit of ensemble size, $N\to\infty$, the prior distribution GMM converges to the empirical distribution,
\begin{equation}
    p(x^b_i) = \frac{1}{N}\sum_{j=1}^N u_{i,j} \delta_{x^b_i - \bar{x}^b_{i,j}},
\end{equation}
which converges weakly to the underlying prior distribution.
Then as the prior Kernel covariance tends towards zero, the posterior GMM estimate defined by~\cref{eq:posterior-Gaussian-mixture} and \cref{eq:posterior-update} converges to the empirical measure,
\begin{equation}
    p(x^a_i) = \frac{1}{N}\sum_{j=1}^N \left(\sum_{k=1}^M w_{i,j,k}\right) \delta_{x^a_i - \bar{x}^b_{i,j}},
\end{equation}
which converges weakly to the exact posterior distribution.
Then, the EnGMF resampling in \cref{pro:resampling-procedure} makes the EnGMF converge to an SIR filter.
\end{proof}

The rate of convergence in~\cref{thm:EnGMF-convergence},
is determined by the free parameters in the GMM that describes the prior distribution~\cref{eq:prior-Gaussian-mixture}, namely the covariance matrices. For a deeper look into Kernel density estimation and how the covariance matrices determine the strength of the approximation of the exact distribution see~\cite{silverman2018density}. 

\section{Adaptive ensemble Gaussian Mixture Filter}
\label{sec:AEnGMF}

Following the observations provided by~\cref{thm:EnGMF-convergence}, we want to choose covariance matrices $\*B_{i,j}$ found in the prior GMM assumption~\cref{eq:prior-Gaussian-mixture} in an intelligent and adaptive manner such that the convergence properties are satisfied.
We additionally attempt to fulfill a desire useful to the practitioner: that practical convergence is achieved with as small as possible number of particles.

To that end, in this work we explore arbitrary parameterized covariance matrices in the prior GMM~\cref{eq:prior-Gaussian-mixture},
\begin{equation}\label{eq:parameterized-covariance-distribution}
    p(x_i | \theta_i) = \sum_{j=1}^N u_{i,j} \!N\left(x_i | x^b_{i,j}, \*B_{i,j}^b(\theta_i)\right),
\end{equation}
where each covariance $\*B_{i,j}^b(\theta_i)$ is a matrix function of some (small number of) parameters $\theta_i$.

The aim of the parameterization in~\cref{eq:parameterized-covariance-distribution} is to find a set of parameters $\theta_i$ that can both be chosen adaptively at each step, and can ensure that the EnGMF satisfies the convergence properties of~\cref{thm:EnGMF-convergence}. 

We now provide a way by which we can solve for the optimal parameters $\theta_i$ in~\cref{eq:parameterized-covariance-distribution} through the expectation maximization algorithm.

\subsection{Expectation Maximization}
\label{sec:EM}

The expectation maximization (EM) algorithm~\cite{bocquet2020bayesian, bishop2006pattern} finds the set of the parameters $\theta_i$ that maximize the conditional distribution of the parameters given the observations $p(\theta_i | y_i)$, at time index $i$. 

Given some initial set of parameters $\theta^{(0)}_i$, the expectation maximization algorithm proceeds in an iterative fashion in two steps:

The \textit{expectation step},
\begin{equation}\label{eq:expectation-step}
   \mathbb{E}_{x^b_i| y_i, \theta^{(m)}_i} \log p(x^b_i, y_i, \theta_i)
\end{equation}
constructs the function representing the expectation of the joint distribution of the prior state, the observations, and the parameters. 
The joint distribution in~\cref{eq:expectation-step} can be written in terms of the prior~\cref{eq:prior-Gaussian-mixture}, observation~\cref{eq:observation-Gaussian-mixture}, and parameter distributions as,
\begin{equation}\label{eq:full-joint-distribution}
    p(x^b_i, y_i, \theta_i) = p(y_i|x^b_i,\theta_i)p(x^b_i|\theta_i)p(\theta_i),
\end{equation}
and, as the observation GMM~\cref{eq:observation-Gaussian-mixture} is not dependent on the parameters,
\cref{eq:full-joint-distribution} can be simplified to,
\begin{equation}\label{eq:simplified-joint-distribution}
    p(x^b_i, y_i, \theta_i) = p(y_i|x^b_i)p(x^b_i|\theta_i)p(\theta_i),
\end{equation}
where the prior distribution~\cref{eq:parameterized-covariance-distribution}
is parameterized in terms of its covariance~\cref{eq:parameterized-covariance-distribution}, and the parameter distribution
\begin{equation}\label{eq:parameter-distribution}
    p(\theta_i),
\end{equation}
is determined on a case-by-case basis. 

The \textit{maximization step} aims to find the value of the parameters $\theta_i$, that maximize the log joint distribution~\cref{eq:simplified-joint-distribution},
\begin{equation}\label{eq:maximization-step}
    \theta^{(m+1)}_i = \argmax_{\theta_i} \mathbb{E}_{x^b_i|y_i,\theta^{(m)}_i} \left[\log p(x^b_i | \theta_i) + \log p(\theta_i)\right],
\end{equation}
where $\theta^{(m)}_i$ are the parameters from the previous step,  $x^b_i | y_i, \theta^{(m)}_i$ are samples from the posterior distribution~\cref{eq:posterior-Gaussian-mixture} given the previous set of parameters $\theta^{(m)}_i$, and the term $\log p(y_i|x^b_i)$ is constant and thus can be safely ignored due to the fact that it does not influence the optimization problem. 
Recall~\cref{rem:posterior-arbitrary-sampling} that in the EnGMF, it is possible to generate an unlimited number of i.i.d. samples from the posterior distribution, thus the maximization step~\cref{eq:maximization-step} can be computed to an arbitrary level of accuracy, given some reasonable assumptions on the distribution of the sample mean.

\begin{remark}[Invertible Covariances]\label{rem:invertible-covariances}
Note that the prior covariance $p(x^b_i | \theta_i)$ in~\cref{eq:maximization-step} requires that the covariance matrices $\*B_{i,j}^b(\theta_i)$ in~\cref{eq:parameterized-covariance-distribution} are invertible, which is not necessarily required by the standard EnGMF. 
\end{remark}

\subsubsection{Stochastic Optimization}

The maximization step~\cref{eq:maximization-step} requires the solution of a stochastic optimization problem. 
Much of the recent literature on stochastic optimization has been focused on machine learning applications~\cite{aggarwal2018neural}. As the number of parameters in $\theta_i$ is small, it is possible to take advantage of methods that are built for the small parameter size case and that differ from typical machine learning optimization methods.
Thus, in this work we utilize a variant of Newton's method.

We can write the loss in the maximization step~\cref{eq:maximization-step} as,
\begin{equation}\label{eq:loss-function}
    \!L(x_i, \theta_i) = \mathbb{E}_{x^b_i|y_i,\theta_i^{(m)}} \left[\log p(x_i | \theta_i) + \log p(\theta_i)\right],
\end{equation}
where the posterior can be written as the following,
\begin{equation}
    x^{a, (m)}_i = x^b_i|y_i,\theta_i^{(m)},
\end{equation}
which is useful shorthand for the following derivations.

One algorithm for finding the maximum of the loss function~\cref{eq:loss-function} is Newton's method, 
\begin{equation}\label{eq:stochastic-Newtons}
\resizebox{0.91\hsize}{!}{$
    \theta_i^{(m+1, p+1)} = \theta^{(m+1, p)} + \alpha_m \left(\mathbb{E}_{x^{a, (m)}_i}[\nabla^2_\theta \!L(x, \theta_i^{(m+1, p)})]\right)^{-1} \mathbb{E}_{x^{a, (m)}_i}[\nabla_\theta \!L(x, \theta_i^{(m+1, p)}))],
    $}
\end{equation}
where $\alpha_m$ is the step-size (also known as the learning rate in the machine learning community), and the initial parameter value for the algorithm is $\theta_i^{(m+1, 1)} \coloneqq \theta_i^{(m)}$, which is the parameter from the previous maximization step~\cref{eq:maximization-step}. 
As it is challenging to compute the expected values in~\cref{eq:stochastic-Newtons} analytically, some sort of approximation procedure is required.

As this is a stochastic optimization procedure, the two expected values in~\cref{eq:stochastic-Newtons} cannot be calculated with the same random samples, as that would introduce unintended bias and variance into the update. In this work we utilize the sub-sampled version of Newton's method~\cite{roosta2019sub} built specifically to handle this scenario.
In sub-sampled Newton's method, independent samples of $x^{a, (m)}_i$ are used to approximate the Hessian $\mathbb{E}_{x^{a, (m)}_i}[\nabla^2_\theta \!L(x, \theta_i^{(m+1, p)})]$ and the gradient $\mathbb{E}_{x^{a, (m)}_i}[\nabla_\theta \!L(x, \theta_i^{(m+1, p)})]$. If the number of samples used is identical, then the number of samples required is double that of the  stochastic gradient descent (SGD) algorithm which only requires the computation of $\mathbb{E}_{x^{a, (m)}_i}[\nabla_\theta \!L(x, \theta^{(m+1,p)})]$. 
As Newton's method achieves faster convergence than  SGD, it is the authors' belief that for this particular scenario the benefits of this approach outweight the additional costs.

\begin{remark}[Quasi-Newton Methods]
Instead of computing an estimate to the Hessian $\mathbb{E}_{x^{a, (m)}_i}[\nabla^2_\theta \!L(x, \theta_i^{(m+1, p)})]$ at every step, it is possible to only compute the Hessian at the initial step $\mathbb{E}_{x^{a, (m)}_i}[\nabla^2_\theta \!L(x, \theta_i^{(m+1, 1)})]$ and use this approximation for all subsequent steps. This type of computationally efficient computation is a type of Quasi-Newton method~\cite{nocedal1999numerical} that is often used in practical applications.
\end{remark}

\begin{remark}[Alternative Optimization Algorithms]
Alternative stochastic optimization algorithms could also be utilized. The classic stochastic gradient descent algorithm~\cite{robbins1951stochastic} is an alternative which would require a smaller step-size $\alpha_m$. Another alternative is ADAM~\cite{kingma2014adam} which would require to keep track of separate momentum and velocity terms.
\end{remark}

\begin{remark}[Incremental Expectation Maximization]\label{rem:incremental-EM}
If the expectation maximization algorithm is performed online in sequential data assimilation, it is not necessary to perform many steps of either the expectation maximization algorithm, or sub-sampled Newton's method~\cref{eq:stochastic-Newtons}.
In this work we initialize the parameters expectation maximization algorithm~\cref{sec:EM} from the previous time step of the data assimilation algorithm. This can be weakly justified as a type of incremental expectation maximization~\cite{neal1998view}, and in the authors' experience significantly increases the utility of the proposed approach.
\end{remark}

We now discuss several different strategies for parameterizing the kernel covariance~\cref{eq:parameterized-covariance-distribution}.

\subsection{Bandwidth-based covariance}

In kernel density estimation, choosing the optimal covariance matrices has had considerable research interest~\cite{silverman2018density}. And, as discussed in~\cref{sec:background} has a considerable impact on the efficacy of the EnGMF algorithm.

One particular case of the covariance in the prior GMM~\cref{eq:parameterized-covariance-distribution} that is a focus of this work is,
\begin{equation}\label{eq:prior-covariance-with-bandwidth}
   \*B^b_{i,j}(\beta^2_{i,N}) = \beta^2_{i,N}\*P^b_{i,N},\quad 1\leq j \leq N,
\end{equation}
where $\beta^2_{i,N}$ is known as the bandwidth parameter~\cite{silverman2018density}, and 
\begin{equation}\label{eq:empirical-covariance}
    \*P^b_{i,N} = \frac{1}{N-1}\*X^b_i\left(I_N - \frac{1}{N}\1\1^T\right)\*X^{b,T}_i \approx \mathbb{E}\left[\left(x^b_i - \mathbb{E}[(x^b_i]\right)\left(x^b_i - \mathbb{E}[(x^b_i]\right)^T\right],
\end{equation}
is known as the empirical covariance, and is an estimate of the covariance matrix of the prior state $x^b_i$. 
In \cref{eq:prior-covariance-with-bandwidth}, the only parameter is the bandwidth estimate, $\theta_i = \beta^2_{i,N}$. 

\begin{remark}[Stochastic Newton's for the Bandwidth Parameter]
When the sub-sampled version of the stochastic Newton's method~\cref{eq:stochastic-Newtons} is applied to the covariance parameterized by the bandwidth parameter~\cref{eq:prior-covariance-with-bandwidth}, then both the stochastic estimate of the gradient and the stochastic estimate of the Hessian are scalars. This enables the computation of the maximization step~\cref{eq:maximization-step} to be performed with minimal linear system solves. 
\end{remark}

The prior kernel covariance estimate in~\cref{eq:prior-covariance-with-bandwidth} takes advantage of the underlying covariance of the data, and is thus a type of online estimate however, the resulting accuracy of the density estimate is still highly dependent on the bandwidth parameter $\beta^2_{i,N}$.

It is known from~\cite{silverman2018density} that if the underlying exact prior distribution of $x^b$ in \cref{eq:prior-Gaussian-mixture} is Gaussian, that the optimal choice of bandwidth parameter $\beta^2$ in \cref{eq:prior-covariance-with-bandwidth} that minimizes the mean integrated square error is,
\begin{equation}\label{eq:Silverman's-rule}
    \beta^2_{i,N,\text{Gaussian}} = \left(\frac{4}{N(n + 2)}\right)^{\frac{2}{n + 4}},
\end{equation}
which is also known as Silverman's rule of thumb. 

In practice most probability distributions of interest are not Gaussian, and~\cref{eq:Silverman's-rule} can result in a very poor approximation of the underlying density~\cite{silverman2018density}, thus a more refined choice of the bandwidth parameter is required.

\Cref{thm:EnGMF-convergence} showed that a sufficient condition for the convergence of the EnGMF is that the covariance estimate tends towards zero as $N\to\infty$.
We now show a condition on the bandwidth parameter that is sufficient for the EnGMF to converge.

\begin{lemma}\label{lem:shrinking-bandwidth}
Without proof, given the sequence of random variables $\{\beta^2_{i,N}\}_{N = 1}^\infty$ parameterized by the particle amount $N$, a sufficient condition for the covariance estimate \cref{eq:prior-covariance-with-bandwidth},
\begin{equation}
    \*B_{i,N} = \beta^2_{i,N}\*P^b_N,
\end{equation}
to tend towards zero in the limit of particle number,
\begin{equation}
    \lim_{N\to\infty}\*B_{i,N} = \*0,
\end{equation}
is that the bandwidth parameter tends towards zero,
\begin{equation}
    \lim_{N\to\infty}\beta^2_{i,N} = 0,
\end{equation}
in the limit of particle number.
\end{lemma}
\begin{corollary}
The sequence $\{\beta^2_{i,N,\text{Gaussian}}\}_{N=1}^\infty$ of bandwidth parameters defined by Silverman's rule of thumb~\cref{eq:Silverman's-rule} satisfies the conditions of \cref{lem:shrinking-bandwidth}.
\end{corollary}

\Cref{lem:shrinking-bandwidth} showed that when $\beta^2_{i,N}$ is a constant, and converges to zero in the limit of particle number $N\to\infty$, the EnGMF converges. However, as we have uncertainty about the bandwidth parameter, it is natural to think about it as a random variable with some distribution. Thus an important choice is that of the distribution of the bandwidth parameter. Care must be taken to ensure that this choice is sufficient to make the resulting algorithm converge.

We provide a sufficient condition on the distribution of the bandwidth parameter $\beta^2_{i,N}$ from \cref{eq:prior-covariance-with-bandwidth}, as an extension of \cref{thm:EnGMF-convergence}. 
We therefore extend~\cref{lem:shrinking-bandwidth} to bandwidth parameters that are random variables with some prior distribution in the expectation maximization algorithm.

\begin{theorem}\label{thm:shrinking-bandwidth}
Given the sequence of random variables $\{\beta^2_{i,N}\}_{N = 1}^\infty$ with a sequence of distributions $\{p(\beta^2_{i,N})\}_{N = 1}^\infty$ parameterized by the particle amount $N$, a sufficient condition for the covariance estimate \cref{eq:prior-covariance-with-bandwidth},
\begin{equation}
    \*B_N = \beta^2_{i,N}\*P^b_N,
\end{equation}
to tend towards zero in distribution in the limit of particle number,
\begin{equation}
    \*B_N \xrightarrow[]{D} \*0,
\end{equation}
is that the distribution of the bandwidth parameter tends towards the delta distribution around zero,
\begin{equation}
    \lim_{N\to\infty} p(\beta^2_{i,N}) = \delta_{0},
\end{equation}
ensuring that $\beta^2_{i,N}$ almost surely becomes $0$.
\end{theorem}
\begin{proof}
If the distribution of $\beta^2_{i,N}$ converges to  $\delta_0$, then the solution to the maximization step in the EM algorithm~\cref{eq:maximization-step} almost surely becomes a constant, namely that $\beta^2_{i,N}\xrightarrow[]{\text{a.s.}} 0$, as required.
\end{proof}

\subsubsection{Choosing the bandwidth distribution}
\label{sec:bandwidth-distribution-choice}

One way in which the conditions of \cref{thm:shrinking-bandwidth} could be satisfied is through an intelligent choice of the probability distribution of the bandwidth parameter $p(\beta^2_{i,N})$.

A common choice in the literature, the principal of maximum entropy (PME)~\cite{jaynes2003probability} could be used to find a good candidate for this distribution.
If we assume that the expected value of the bandwidth, $\beta^2_{i,N}$, is Silverman's rule of thumb~\cref{eq:Silverman's-rule}, and we have no other information available, then the distribution that satisfies the PME is the exponential distribution,
\begin{equation}\label{eq:exponential-distribution}
    p(\beta^2_{i,N}) = \beta^{-2}_{i,N,\text{Gaussian}}e^{-\beta^{-2}_{i,N,\text{Gaussian}}\beta^2_{i,N}},
\end{equation}
this distribution, however, always has a single mode at zero, thus, from the authors' experience, is ill-suited for use in expectation maximization.

It is possible to perform some slight-of-hand in order to make this assumption more tractable. It is more efficient to look at $\beta_{i,N}$, the square-root of the bandwidth parameter. If~\cref{eq:exponential-distribution} is the distribution of $\beta^2_{i,N}$, then $\beta_{i,N}$ is distributed according to
\begin{equation}\label{eq:Rayleigh-distribution}
    p(\beta_{i,N}) = 2  \beta^{-2}_{i,N,\text{Gaussian}}  \beta_{i,N} e^{-\beta^{-2}_{i,N,\text{Gaussian}}\beta^2_{i,N}}
\end{equation}
which is the Rayleigh distribution~\cite{papoulis1965probability} with known mode $2^{-1/2}\beta_{i,N,\text{Gaussian}}$. We assume this Rayleigh distribution~\cref{eq:Rayleigh-distribution} on the bandwidth parameter for the remainder of this paper.

It is also possible to assume a more general distribution around $\beta^2_{i,N}$, such as a gamma distribution, though this choice would introduce another free parameter into the algorithm; an undesirable outcome.

While the parameterized covariance in \cref{eq:prior-covariance-with-bandwidth} is well-studied, it has a few limitations that prevent it from being used in high-dimensional inference, chief along those being the fact that the covariance estimate in~\cref{eq:empirical-covariance} can potentially be low-rank, and thus generate a covariance that is not invertible~\cref{rem:invertible-covariances}, thus we can introduce covariance matrix estimates that have extra parameters in order to mitigate this issue.

\subsection{Covariance Shrinkage Estimates}

Covariance shrinkage~\cite{chen2009shrinkage,chen2010shrinkage,chen2011robust,chen2012shrinkage,ledoit2004well} aims to use extra prior information about the covariance of $x^b_i$ in~\cref{eq:prior-Gaussian-mixture} in order to have a more accurate covariance estimate in the case when the number of samples is smaller than the dimension of the dynamical system $N < n$. Covariance shrinkage methods have previously been employed for ensemble data assimilation~\cite{nino2015ensemble,popov2020stochastic} and for regularization in particle filtering~\cite{popov2022stochastic}.

Assume that we have prior information about the covariance structure of $x^b_i$ in the form of a `target' covariance matrix $\*T_i$.
The covariance shrinkage estimate to the covariance, scaled by the bandwidth, is given by,
\begin{equation}\label{eq:shrinkage-estimate}
    \*B^b_{i,j} = \beta^2_{i,N}\left[ \gamma_i\mu_i\*T_i + (1 - \gamma_i) \*P^b_i\right]
\end{equation}
where,
\begin{equation}\label{eq:shrinkage-mu-C}
    \mu_i = n^{-1}\tr\*C_i,\quad \*C_i = \*T_i^{-\frac{1}{2}}\*P^b_i\*T_i^{-\frac{1}{2}}
\end{equation}
is a rescaling factor,
and $\gamma_i$ is the shrinkage factor, which we treat as a parameter.

Under Gaussian assumptions on the samples, $x^b_{i,j}$, a good known shrinkage factor is,
\begin{equation}\label{eq:gamma-RBLW}
\begin{aligned}
    \gamma_{i,\text{RBLW}} &= \min\left[\frac{N-2}{N(N+2)} + \frac{(n+1)N - 2}{N(N + 2)(n - 1) \hat{U}_i} , 1\right],\\ \hat{U}_i &= \frac{1}{n - 1}\left(\frac{n\tr \*C^2_i}{\tr^2\*C_i} - 1\right)
\end{aligned}
\end{equation}
called the Rao-Blackwell Ledoit-Wolf estimator~\cite{chen2009shrinkage}.

One possible choice of the target matrix $\*T_i$ that does not require any prior knowledge is the diagonal of the empirical covariance~\cref{eq:empirical-covariance},
\begin{equation}\label{eq:target-is-diagonal}
    \*T_i = \bdiag\*P^b_i,
\end{equation}
where the notation of $\bdiag$ is introduced to signify the matrix consisting of only the diagonal of the subsequent term.
Observe that,
\begin{equation}
    \tr\left[\left(\bdiag\*P^b_i\right)^{-\frac{1}{2}} \*P^b_i\left(\bdiag\*P^b_i\right)^{-\frac{1}{2}}\right] = n,
\end{equation}
therefore when the target matrix is defined by~\cref{eq:target-is-diagonal}, $\mu_i = 1$, and calculating $\*C_i$, from \cref{eq:shrinkage-mu-C}, only need to be performed when calculation $\gamma_{i,\text{RBLW}}$ in~\cref{eq:gamma-RBLW}.

\begin{remark}[On $p(\gamma_i)$]
A commonly made assumption is that parameters are independently distributed, therefore the distribution of $p(\beta^2_{i,N})$ can be chosen independently of the distribution $p(\gamma_i)$.
As there are no requirements that the optimal $\gamma_i$ is dependent on ensemble size, it is a natural choice to assume a uniform likelihood,
\begin{equation}
    p(\gamma_i) \propto 1, 
\end{equation}
which is a typical assumption  in parameter estimation~\cite{bocquet2020bayesian}.
\end{remark}

\subsection{Covariance Localization}

In the geosciences, states usually have some sort of innate spatial structure. State variables that are spatially far apart are generally more weakly correlated than states that are closer together. Taking advantage of this fact, covariance localization~\cite{asch2016data} is a matrix tapering technique which aims to reduce the impact of spurious correlations caused by undersampled $(N \ll n)$ covariance matrix estimates.

In this work we focus on what is known as the B-localization methodology, and combine it with the bandwidth scaling~\cref{eq:prior-covariance-with-bandwidth} in the following manner,
\begin{equation}\label{eq:B-localization-adaptive-estimate}
    \*B^b_{i,j} = \beta^2_{i} \left(\rho(r_i) \circ \*P^b_{i}\right),
\end{equation}
where the matrix $\rho(r_i)$ contains a set of decorrelation variables parameterized by the localization radius $r_i$, and $\circ$ is the element-wise Schur product.

A common choice for $\rho$ is known as Gaussian localization,
\begin{equation}
    \rho(r_i)_{\ell,q} = e^{-\frac{1}{2}\frac{d(\ell,q)^2}{r_i^2}},
\end{equation}
where $d(\ell,q)$ represents the spatial distance between the variables at index $\ell$ and index $q$.

\begin{remark}[Choice of Localization Radius $r_i$]\label{rem:localization-radius-choice}
The choice of localization radius $r_i$ in~\cref{eq:B-localization-adaptive-estimate} can be informed by the temporal covariance of the model of interest if the model of interest is Ergodic, however in practice, the best localization radius is almost always determined empirically.  
\end{remark}

Adaptive-in-time choices for $r_i$ have been explored for the ensemble Kalman filter in~\cite{popov2019bayesian}.

\subsection{Practical Implementation of the AEnGMF}

\begin{figure}
    \centering
    \begin{tikzpicture}
        \node[align=center] at (-1.8,2) {Candidate\\Prior};
        \node (A) at (0,0) {
        \begin{axis}[xshift=-5cm,
          axis x line=none,
          axis y line=none,
          domain=-6:18,
          samples=100,
          xticklabels=\empty,
          width=6cm, height=3cm,
          every axis plot/.append style={line width=2pt, mark size=3.5pt},
        ]
        \addplot [tolblue] {exp(-(x^2)/4)};
        \addplot [tolred] {exp(-((x-6)^2)/4)};
        \addplot [tolyellow] {exp(-((x-12)^2)/4)};
        \end{axis}};
        \node[align=center] at (3.2,2) {Candidate\\Posterior};
        \node (B) at (1,0) {
        \begin{axis}[
          axis x line=none,
          axis y line=none,
          domain=-6:14,
          samples=100,
          xticklabels=\empty,
          width=4cm, height=3cm,
          every axis plot/.append style={line width=2pt, mark size=3.5pt},
        ]
        \addplot [tolblue] {exp(-(x^2)/4)};
        \addplot [tolred] {exp(-((x-4)^2)/4)};
        \addplot [tolyellow] {exp(-((x-8)^2)/4)};
        \end{axis}};
        \draw[->, line width=2pt] (0,0.5) -- (2,0.5) node[midway,above] {EnGMF};
        \draw[->,line width=2pt] (3.2,-0.1) -- (3.2,-1) -- (-1.8,-1) node[midway,below] {Expectation Maximization} -- (-1.8,-0.1);
    \end{tikzpicture}
    \caption{An illustration of the outer loop of the AEnGMF algorithm. In this example the candidate prior is made up of three distinct Gaussian modes which are plotted separately. The candidate prior is transformed into the candidate posterior through the standard EnGMF update~\cref{eq:posterior-update}. Next the expectation maximization algorithm~\cref{eq:expectation-step,eq:maximization-step} is performed, and a new candidate prior is obtained. This procedure repeats until some desired level of convergence is achieved.}
    \label{fig:AEnGMF-outer-loop-illustration}
\end{figure}

\begin{algorithm}[t]
\hspace*{\algorithmicindent} \textbf{Input} Initial ensemble $\*X^a_0$, initial estimate for the parameters $\theta_0$, outer loop iteration count $M$, inner loop iteration count $P$, learning rate $\alpha$, and number of internal samples $S$.\\
\begin{algorithmic}
\For{$i = 1, \dots$}
\State{\% Propagate the ensemble forward in time through the model}
\State{$\*X_i^b = \!M(\*X^a_{i-1})$}
\State{\% Initialize the parameters $\theta$ to the parameters from the previous step}
\State{$\theta_i^{(1,1)}\coloneqq \theta_{i-1}$}
\State{\% Perform the expectation maximization loop $M$ times }
    \For{$m = 1,\dots, M$}
        \State{\% Construct loss function}
        \State{$\!L(x, \theta) \coloneqq \mathbb{E}_{x^b|y,\theta_i^{(m)}} \left[\log p(x | \theta) + \log p(\theta)\right]$}
        \State{\% Initialize the inner loop $\theta$ parameter}
        \State{$\theta_i^{(m+1, 1)} \coloneqq \theta_i^{(m, P+1)}$}
        \State{\% Perform $P$ steps of subsampled Newton's method.}
        \For{$p = 1,\dots, P$}
            \State{\% Sample $S$ particles from the candidate posterior.}
            \State{$\*X^a \sim_{(S)} \pi(x | y, \theta_i^{(m)})$}
            \State{\% Compute the loss gradient}
            \State{$g\coloneqq \mathbb{E}_{\*X^a}[\nabla_{\theta} \!L(X, \theta_i^{(m+1, p)})]$}
            \State{\% Similarly, compute sample Hessian using different samples}
            \State{$\*X^a \sim_{(S)} \pi(x | y, \theta_i^{(m)})$}
            \State{$H \coloneqq \mathbb{E}_{\*X^a}[\nabla^2_{\theta} \!L(X, \theta_i^{(m+1, p)})]$}
            \State{\% Compute new estimate of the parameters}
            \State{$\theta_i^{(m+1,p+1)} \coloneqq \theta_i + \alpha H^{-1} g$}
        \EndFor
        \State{\% Set the current $\theta$ parameter}
        \State{$\theta_i^{(m+1)} \coloneqq \theta_i^{(m+1,P+1)}$}
    \EndFor
    \State{\% Set the $\theta$ parameter for the current time index}
    \State{$\theta_i \coloneqq \theta_i^{(M+1)}$}
    \State{\% Sample a new ensemble of $N$ particles with new parameters $\theta_i$}
    \State{$\*X^a_i \sim_{(N)} \pi(x | y, \theta_i)$}
\EndFor

\end{algorithmic}
\caption{The Adaptive Ensemble Gaussian Mixture Filter}
\label{alg:AEnGMF}
\end{algorithm}

We are now able to combine all the elements presented in this section to fully describe the inner-workings of the adaptive ensemble Gaussian mixture filter (AEnGMF).
The AEnGMF operates as follows. First a choice of parameterized covariance is made by the user. This choice determines the parameters that are optimized for. Next a choice of parameter distribution is required. In this work the bandwidth parameter is assumed to be distributed according to the Rayleigh distribution~\cref{sec:bandwidth-distribution-choice}, and the rest of the parameters are assumed to be proportional to one, thus of no additional consequence.

At each step of the algorithm, the previous choice of covariance parameters is carried over, $\theta_i^{(1,1)} \coloneqq \theta_{i-1}$. This choice from~\cref{rem:incremental-EM} is motivated by incremental approaches to expectation maximization, and lends itself particularly well to parameterized covariances do not depend on their parameters changing a lot from step to step.

Next, $M$ iterations of the expectation-maximization algorithm are performed. The expectation maximization algorithm can be treated as the `outer-loop'~\cite{nocedal1999numerical} in this optimization procedure.

The cost function is solved using $P$ loops of sub-sampled Newton's method~\cref{eq:stochastic-Newtons} with a constant learning-rate $\alpha$ making this the `inner-loop' algorithm.
As it is possible to sample from the posterior arbitrarily~\cref{rem:posterior-arbitrary-sampling}, the gradient and Hessian calculations can be performed using a different number of samples, $S$, than that of the number of particles $N$. Specifically, the gradient is computed using $S$ samples from the candidate posterior, and the Hessian is computed using $S$ separate samples from the candidate posterior, for a total of $2S$ samples.

As the AEnGMF is a particle filter, resampling of $N$ particles is performed at the end of the algorithm with the EnGMF resampling procedure~\cref{pro:resampling-procedure}.

The outer loop of the algorithm is illustrated in~\cref{fig:AEnGMF-outer-loop-illustration}, and a detailed step-by-step look at the algorithm can be seen in~\cref{alg:AEnGMF}.

\begin{remark}[Choosing $M$, $P$, $\alpha$, and $S$]\label{rem:choosing-practical-parameters}
In the author's experience, it is much more advantageous to perform multiple iterations of the expectation maximization algorithm than that of sub-sampled Newton's method, thus it is advantageous to take $M \geq P$. It is also advantageous to oversample sample the gradient and Hessian, thus $S \geq N$. 
By far the hardest choice to make is that of the learning-rate $\alpha$. A learning rate that is too large ($\alpha \approx 1$) could cause the algorithm to become unstable and choose bad parameters $\theta$. A learning rate that is too small ($\alpha\to 0$) could lead to parameters that react poorly to the changing conditions of the states. From the practitioner's point of view, this is by far the most important parameter to choose correctly.
\end{remark}

\begin{remark}[Considerations for the High-dimensional Setting]
There are many considerations to be made for getting the AEnGMF to work in the high-dimensional setting.
First is that the computation of the covariance cannot be made explicitly. This can be resolved by utilizing a covariance estimate that does not have to be explicitly computed like that of the shrinkage estimate in~\cref{eq:shrinkage-estimate}. 
Matrix inverse vector products of~\cref{eq:shrinkage-estimate} can also be computed without explicit computation of the entire matrix.
The normalizing factor issue in \cref{rem:normalizing-factor} can also be mitigated by this covariance matrix estimate. Another problem is the resampling procedure in~\cref{pro:resampling-procedure}, which requires the computation of matrix square root vector products. Methods such as those proposed in~\cite{allen2000numerical,guttel2013rational,chow2014preconditioned} could be utilized to solve this issue, though this is still an open problem.
\end{remark}

\section{Numerical Experiments}
\label{sec:numerical-experiments}

The aim of the numerical experiments is first to demonstrate the viability and convergence of the proposed AEnGMF on small-scale problem, and secondly to demonstrate the more complicated covariance parameterization approaches on a larger-scale problem.

\subsection{Lorenz '63}

\pgfplotsset{clean/.style={axis lines*=left,
        axis on top=true,
        axis x line shift=0.0em,
        axis y line shift=0.75em,
        every tick/.style={black, thick},
        axis line style = ultra thick,
        tick align=outside,
        clip=false,
        major tick length=4pt}}

\begin{figure}[t]
    \centering
    \begin{tikzpicture}
    \begin{semilogxaxis}[clean,
        cycle list name=tol,
        xmode=log,
        log ticks with fixed point,
        xtick=data,
        table/col sep=comma,
        xmin = 23,
        xmax = 550,
        ymin = 2.65,
        ymax = 6.75,
        xlabel = {Number of Particles (N)},
        ylabel = {Mean Spatio-temporal RMSE},
        every axis plot/.append style={line width=2pt, mark size=3.5pt},
        legend style={at={(1,0.85)},anchor=center},
        legend cell align={left}]
    \addplot[color=black!35,forget plot, dotted] table [x=N,y expr=4.71433, mark=none]{data/L63wOptimalScaling.csv}  node[pos=1, right] {\,\,EnKF ($N\to\infty)$};
    \addplot[color=black!35,forget plot,dotted] table [x=N,y expr=2.81689, mark=none]{data/L63wOptimalScaling.csv}  node[pos=1, right] {\,\,SIR ($N\to\infty)$};
    \addplot table [x=N, y=EnGMFEM, col sep=comma] {data/L63wOptimalScaling.csv};
    \addlegendentry{AEnGMF};
    \addplot table [x=N, y=EnGMFb1, col sep=comma] {data/L63wOptimalScaling.csv};
    \addlegendentry{EnGMF ($\beta^2 = \beta^2_{\text{Gaussian}}$)};
    \addplot table [x=N, y=EnGMFb03, col sep=comma] {data/L63wOptimalScaling.csv};
    \addlegendentry{EnGMF ($\beta^2 = 0.3\beta^2_{\text{Gaussian}}$)};
    \pgfplotsset{cycle list shift=+1};
    \addplot table [x=N, y=POEnKF, col sep=comma] {data/L63wOptimalScaling.csv};
    \addlegendentry{EnKF};
    \end{semilogxaxis}
    \end{tikzpicture}
    \caption{Simulation results for the Lorenz '63 equations for four different data assimilation algorithms. The blue line with circular marks represents the AEnGMF, the red line with square marks represents the canonical EnGMF with Silverman's rule of thumb, the yellow line with x marks represents the EnGMF with a scaled Silverman's rule of thumb, and the Raspberry line with  diamond marks represents the EnKF. Two baseline lines, running the EnKF and a particle filter (SIR) for a high particle number are also provided to provide theoretical bounds.}
    \label{fig:lorenz63-experiment}
\end{figure}
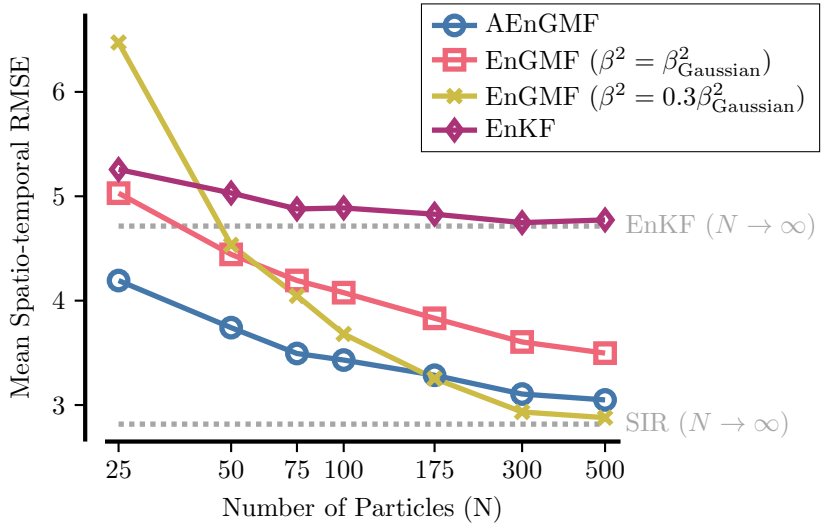

With the first set of experiments we aim to look at a highly non-linear system with a non-linear observation operator. We focus on the standard EnGMF case of the Kernel covariance parameterized by the bandwidth parameter~\cref{eq:prior-covariance-with-bandwidth}.

We take the the 3-variable Lorenz '63 equations~\cite{lorenz1963deterministic},
\begin{equation}\label{eq:lorenz63-equations}
\begin{aligned}
    x' &= \sigma(y - x),\\
    y' &= x(\rho - z),\\
    z' &= xy-\beta z,
\end{aligned}
\end{equation}
with canonical parameters $\sigma = 10$, $\rho = 28$, and $\beta = \frac{8}{3}$. The time between assimilations is taken to be $\Delta t = 0.5$, which allows the system enough time to evolve in a highly non-linear manner.

It is known~\cite{hirsch2012differential} that the system~\cref{eq:lorenz63-equations} has three critical points, one at the origin, and in the center of each of the butterfly wings. The first one of the critical points,
\begin{equation}\label{eq:lorenz-3-critical-point}
\begin{aligned}
        x_c &= \sqrt{\beta(\rho - 1)},\\
        y_c &= \sqrt{\beta(\rho - 1)},\\
        z_c &= \rho - 1,
\end{aligned}
\end{equation}
defines the center of one of the wings of the butterfly, with the other center being $(-x_c, -y_c, z_c)$ and the origin being $(0,0,0)$.
For the non-linear observation operator we measure the distance from the critical point~\cref{eq:lorenz-3-critical-point} to the point being measured, %
\begin{equation}
\begin{aligned}
\!H\left(x, y, z\right) = \sqrt{(x - x_c)^2 + (y - y_c)^2 + (z - z_c)^2},
\end{aligned}
\end{equation}
as a scalar observation, with Gaussian error with an error variance of $\*R = 1$. 

The goal of this experiment is to show that the various variants of the EnGMF are superior to the ensemble Kalman filter (EnKF) and converge to exact Bayesian inference in the limit of particle number ($N\to\infty$). We therefore calculate two reference boundaries for this problem, one using the EnKF for a large ensemble size (N = 1000) and for the sequential importance resampling (SIR) particle filter with a large number of particles (N = 1000), specifically the variant found in~\cite{reich2015probabilistic}.

It is known in the literature~\cite{silverman2018density, janssen1995scale} that Silverman's rule-of-thumb is usually an over-estimate of the optimal bandwidth term. We thus attempt to find a scaling factor $0 < s < 1$ such that the bandwidth parameter defined by the product,
\begin{equation}
    \beta^2_{i,N} = s \beta^2_{\text{Gaussian}},
\end{equation}
would produce the minimal error for our choice of number of particles. From a quick parameter sweep, it was determined that $s = 0.3$ provides a good factor, that is optimal for a high number of particles.

Thus, we run four different algorithms, the EnKF, the AEnGMF, the EnGMF with Silverman's rule of thumb, and the EnGMF with Silverman's rule of thumb scaled by $s = 0.3$, for this model setup for various numbers of particles ranging from $N = 25$ to $N = 500$.

All experiments were run for four independent initial ensembles for $5500$ assimilation steps with the first $500$ discarded for spinup, meaning that the first $500$ steps do not count into the error calculations to let the filter reach a steady state. The mean of the spatio-temporal RMSE,
\begin{equation}\label{eq:RMSE}
    \text{RMSE}(\bar{x}, x^t) = \sqrt{\frac{1}{nT}\sum_{i=1}^T\sum_{l=1}^n \left(\bar{x}_{i,l} - x^t_{i,l}\right)^2},
\end{equation}
is calculated between the statistical mean $\bar{x}$ and the truth $x^t$, over the four runs and is the metric by which the efficacy of the algorithms is determined.

For the choices of parameters in~\cref{alg:AEnGMF}, we choose $M = 5$ loops of the expectation maximization algorithm, $P=1$ loops of sub-sampled Newton's method, sampling $S = N$ exactly as many samples as there are particles, and a high learning rate of $\alpha = 1$. The Rayleigh distribution with mean of Silverman's rule-of-thumb is chosen for the bandwidth parameter just like in~\cref{sec:bandwidth-distribution-choice}.

The results of the experiments are visually demonstrated in~\cref{fig:lorenz63-experiment}. As can be seen, the AEnGMF is consistently lower in error than the EnGMF with bandwidth equivalent to Silverman's rule-of-thumb, and provides lower error in the particle number range of $N=25$ to $N = 100$. The EnGMF with scaling factor $s = 0.3$ is the only algorithm to perform worse than the EnKF for $N=25$, but also produces the lowest error between $N = 300$ and $N = 500$.

It can be seen that the AEnGMF converges to true Bayesian inference in the limit of particle number ($N\to\infty$), which is a strong confirmation of the result identified by \cref{thm:shrinking-bandwidth}.
Although this is a strong result, the practical use of the Rayleigh distribution~\cref{sec:bandwidth-distribution-choice} for achieving this effect can be questioned as the convergence is clearly sub-optimal.
A better choice of the distribution of the bandwidth parameter is required.

\subsection{Lorenz '96}

\begin{figure}[t]
    \centering
    \begin{tikzpicture}
    \begin{axis}[clean,
        cycle list name=tol,
        xtick=data,
        table/col sep=comma,
        xmin = 3,
        xmax = 42,
        ymin = -0.2,
        ymax = 5.5,
        xlabel = {Number of Particles (N)},
        ylabel = {Mean Spatio-temporal RMSE},
        every axis plot/.append style={line width=2pt, mark size=3.5pt},
        legend style={at={(1,0.8)},anchor=center},
        legend cell align={left}]
    \addplot[color=black!35,forget plot, dotted] table [x=N,y expr=0.2598, mark=none]{data/L96data.csv}  node[pos=1, right] {\,\,EnKF ($N\to\infty)$};
    \addplot table [x=N, y=EnGMFEM, col sep=comma] {data/L96data.csv};
    \addlegendentry{Shr-AEnGMF};
    \addplot table [x=N, y=EnGMF, col sep=comma] {data/L96data.csv};
    \addlegendentry{Shr-EnGMF};
    \addplot table [x=N, y=EnGMFEMloc, col sep=comma] {data/L96data.csv};
    \addlegendentry{LAEnGMF};
    \addplot table [x=N, y=EnGMFloc, col sep=comma] {data/L96data.csv};
    \addlegendentry{LEnGMF};
    \addplot table [x=N, y=ETKF, col sep=comma] {data/L96data.csv};
    \addlegendentry{LEnKF};
    \end{axis}
    \end{tikzpicture}
    \caption{Simulation results for the Lorenz '96 equations for five different data assimilation algorithms. The dark blue line with circle marks represents the AEnGMF with a shrinkage-based estimate to the covariance , with the light red line with square represents the standard EnGMF with a shrinkage-based estimate to the covariance, the yellow line with x marks represents a localized AEnGMF, the light-blue line with plus marks represents a localized EnGMF, and the raspberry line with diamond marks represnts the localized EnKF.}
    \label{fig:lorenz96-figure}
\end{figure}
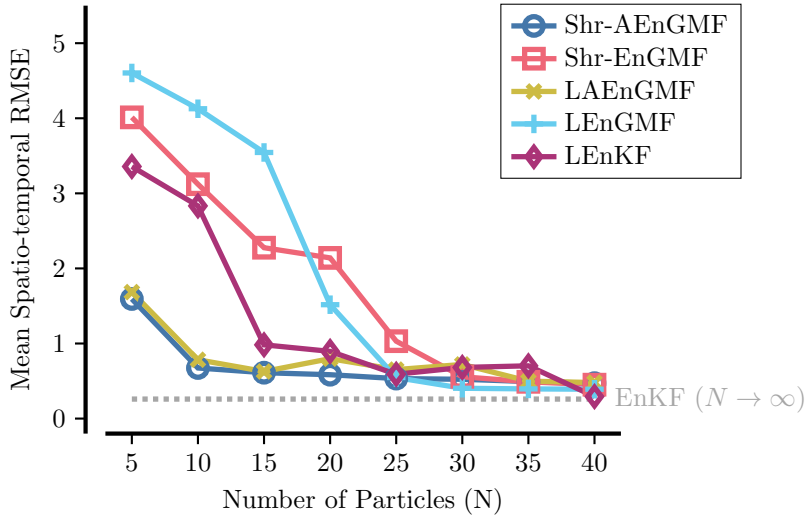

For the next set of experiments, we look at the case of an undersampled ($N \ll n$) estimate of the prior distribution. We look at two different types of covariance matrix parameterizations: one based on covariance shrinkage~\cref{eq:shrinkage-estimate} with the target matrix being the diagonal of the statistical covariance~\cref{eq:target-is-diagonal}, and a localization-based covariance matrix estimate~\cref{eq:B-localization-adaptive-estimate}.

For the model of interest we take the 40-variable Lorenz '96 equations~\cite{lorenz1996predictability},
\begin{equation}
    x_k' = -x_{k-1}(x_{k-2} - x_{k+1}) - x_k + F\dots,\quad k = 1,\dots, 40,
\end{equation}
with cyclic boundary conditions, and the forcing factor $F = 8$. For the time between assimilations we take one day of model time which is equivalent to a $\Delta t = 0.2$, leading to a high level of non-linearity in the system.

We want to compare the AEnGMF approach of adaptively choosing the parameters of the Kernel covariance with that of the more classic EnGMF approach where the parameters are determined by a known good heuristic. We also want to compare with a base-line state-of-the-art algorithm, the localizaed ensemble Kalman filter. To that end, we perform experiments on the following set of filters:
\begin{enumerate}
    \item the shrinkage-based AEnGMF (Shr-AEnGMF), with parameters of $\beta^2_i$ for the bandwidth and $\zeta_i = \tanh^{-1} \gamma_i$, for an unbounded transformation of the shrinkage parameter $0 < \gamma_i < 1$,
    \item the shrinkage-based EnGMF (Shr-EnGMF) with bandwidth defined by Silverman's rule-of-thumb~\cref{eq:Silverman's-rule}, $\beta_i^2 = \beta^2_{i,\text{Gaussian}}$ and the RBLW \cref{eq:gamma-RBLW} shrinkage parameter $\gamma_i = \gamma_{i,\text{RBLW}}$,
    \item the localized AEnGMF (LAEnGMF) with parameters of $\beta_i^2$  for the bandwidth and $\zeta_i = \sqrt{r_i}$ for an unbounded transformation of the localization radius $0 < r_i$,
    \item the localized EnGMF (LEnGMF) with with bandwidth defined by Silverman's rule-of-thumb~\cref{eq:Silverman's-rule}, $\beta_i^2 = \beta^2_{i,\text{Gaussian}}$ and a fixed radius of $r_i = 4$,
    \item and the localized EnKF (LEnKF) with fixed radius $r_i = 4$ for a useful comparison with a state-of-the-art filter.
\end{enumerate}
For the non-linear observation operator, we take the point-wise non-linear operator,
\begin{equation}
    \!H(x_i) = \frac{x_i}{2}\left[1 + \left(\frac{\lvert x_i \rvert}{10}\right)^{\omega - 1}\right], 
\end{equation}
as found in~\cite{asch2016data}, with $\omega = 5$ for a medium level of non-linearity, with the observation covariance matrix being set to $\*R = \frac{1}{4}\*I_{40}$. The number of particles is taken to range from as little as $N = 5$ to as high as $N = 40$.

All experiments were run for four independent initial ensembles for $1100$ assimilation steps with the first $100$ discarded for spinup, meaning that the first $100$ steps do not count into the error calculations to let the filter reach a steady state. For our error metric we again take the spatio-temporal RMSE~\cref{eq:RMSE}.

For the choices of parameters in~\cref{alg:AEnGMF}, we choose $M = 1$ loops of the expectation maximization algorithm, $P=1$ loops of sub-sampled Newton's method, sampling $S = 100$ to sample in the excess, and a low learning rate of $\alpha = 1e-2$. 
The Rayleigh distribution with mean of Silverman's rule-of-thumb is chosen for the bandwidth parameter just like in~\cref{sec:bandwidth-distribution-choice}, with both the radius $r$ and shrinkage parameter $\gamma$ having distributions proportional to one along all of their support as they are not required for convergence.

The results for this round of experiments can be seen in figure~\cref{fig:lorenz96-figure}. 
At around $N = 30$ particles, all algorithms perform roughly the same, thus the interesting behavior occurs when there are less particles. 
Both versions of the EnGMF without adaptive covariance estimates (Shr-EnGMF and LEnGMF) perform worse than the localized EnKF.
The adaptive versions of the same algorithms (Shr-AEnGMF and LAEnGMF) perform significantly better than all other tested algorithms and practically converge for $N=10$ particles.
These results highlight the need and utility of the adaptive covariance estimate approach in the EnGMF presented in this paper.

\section{Conclusions}
\label{sec:conclusion}

By leveraging parameterized sample covariance estimates and the expectation maximization algorithm, this work introduced the adaptive ensemble Gaussian mixture filter (AEnGMF) as an extention of the ensemble Gaussian mixture filter (EnGMF). Theoretical results about the convergence properties of this filter were derived by making assumptions about the distribution of the kernel bandwidth. 
Numerical results have verified the theoretical convergence properties of the AEnGMF, and have shown that for a certain set of parameters the AEnGMF has superior converge to that of the EnGMF.

Future work could extend the AEnGMF to a smoothing~\cite{reich2015probabilistic} framework, a hybrid filtering~\cite{frei2013bridging} framework, and to a multifidelity filtering~\cite{popov2021multifidelity} framework.

An active research direction is in applying the AEnGMF to a real-world orbit tracking problem~\cite{yun2022kernel}.

Work exploring practical consideration on choosing the parameters discussed in~\cref{rem:choosing-practical-parameters} is also of independent interest.

\bibliographystyle{siamplain}
\bibliography{
    bibfiles/covarianceshrinkage,
    bibfiles/stochasticoptimization,
    bibfiles/em, 
    bibfiles/engmf, 
    bibfiles/kernelapproximation, 
    bibfiles/filteringgeneral,
    bibfiles/probability,
    bibfiles/multifidelity,
    bibfiles/problems,
    bibfiles/misc
    }

\end{document}